\documentclass{scrartcl}
\usepackage[utf8]{inputenc} 
\usepackage{amsthm}
\usepackage{amsmath}
\usepackage{amssymb}
\usepackage{dsfont}
\usepackage{verbatim}
\usepackage{todonotes}
\theoremstyle{plain}
\newtheorem{Theorem}{Theorem}[section]
\newtheorem{Lemma}[Theorem]{Lemma}
\newtheorem{Corollary}[Theorem]{Corollary}
\newtheorem{Proposition}[Theorem]{Proposition}
\theoremstyle{definition}

\newtheorem{Notation}[Theorem]{Notation}
\theoremstyle{remark}
\newtheorem{Remark}[Theorem]{Remark}
\newcounter{condition}
\renewcommand{\thecondition}{C\arabic{condition}}
\newenvironment{Condition}{\begin{trivlist}\refstepcounter{condition}\item[ \textbf{Condition}~\textbf{\thecondition.}]\itshape}{\end{trivlist}}
\newcounter{assumption}
\renewcommand{\theassumption}{A\arabic{assumption}}

\newcommand{\dd}[1]{{\operatorname{d}}#1}
\newcommand{\R}{\mathds{R}}

\newcommand{\E}{\mathds{E}}
\newcommand{\N}{\mathds{N}}
\newcommand{\1}{\mathds{1}}

\newcommand{\Co}{\mathcal{C}}

\newcommand{\Prob}{\mathds{P}}
\newcommand{\norm}[1]{\left\lVert #1 \right\rVert}
\newcommand{\abs}[1]{\left| #1 \right|}

\DeclareMathOperator{\sgn}{sgn}

\title{On the Strong Feller Property and Well-Posedness for SDEs with Functional, Locally Unbounded Drift}
\author{ Stefan Bachmann\\  {\small Email: bachmann@math.uni-leipzig.de} \\ \multicolumn{1}{p{.55\textwidth}}{\normalsize\centering\emph{Institut f\"ur Mathematik, Universit\"at Leipzig, Augustusplatz 10, 04109 Leipzig, Germany}}}
\begin{document}
	\maketitle
	\begin{abstract}
		\textbf{Abstract}. We study functional stochastic differential equations with a locally unbounded, functional drift focusing on well-posedness, stability and the strong Feller property. Following the non-functional case, we consider integrability conditions and only need minimal continuity assumptions. Our approach is mainly based on Zvonkin's transformation \cite{Zvonkin} and the convergence concept for random variables in topological spaces in \cite{Stefan2}, which extends the probabilistic approach of Maslowski and Seidler \cite{Maslowski2000}. Our arguments for the strong Feller property are mostly probabilistic, relatively elementary and can still deal with non-regular drifts. This allows extensions in various ways and are applicable in different, more complex situations.
	\end{abstract}
	
	\section{Introduction}
	In this paper, we consider stochastic functional differential equations of the following form 
	\begin{align}
		\label{eq}
		\begin{aligned}
			\dd{X^x}(t)&=B(t,X^x)\dd{t}+\sigma(t,X^x(t))\dd{W}(t)\\
			X_0&=x\in C\left([-r,0],\R^d\right)
		\end{aligned}
	\end{align}
	where $W$ is a $d$-dimensional Brownian motion, $B:\R_{\geq0}\times C(\R_{\geq-r},\R^d)\to\R^d$ is non-anticipating and $\sigma:\R_{\geq0}\times\R^d\to\R^{d\times d}$ is measurable, bounded, non-degenerate and Lipschitz in space.
	
	Non-functional stochastic differential equations (SDEs) with discontinuous drift have been extensively studied: Portenko \cite{Portenko}, Veretennikov \cite{Ver} and Zvonkin \cite{Zvonkin} considered - among other things - well-posedness for SDEs with bounded, discontinuous drift terms. Krylov and R\"ockner have shown existence and uniqueness for locally unbounded drifts and constant, non-degenerate diffusion coefficients in \cite{KrylovRoeckner}. Singular SDEs with non-constant, non-degenerate diffusion matrices have been studied by Mart\'inez, Gy\=ongy \cite{Gyoengy2001} and Zhang \cite{Zhang2011}. Additionally, there are numerous results for the strong Feller property for non-functional, singular SDEs with the Euclidean state space $\R^d$ (i.e. \cite{Zhang2011}).
	
	However, we are interested in the strong Feller property for functional SDEs with the state space of path segments $C\left([-r,0],\R^d\right)$ for some $r>0$. Es-Sarhir, von Renesse and Scheutzow established a Harnack-inequality under Lipschitz conditions and constant, non-degenerate diffusion matrices in \cite{es-sarhir2009}, which implies the strong Feller property. Wang and Yuan proved a log-Harnack inequality for non-constant, non-degenerate diffusion coefficients in \cite{WangYuan}. In \cite{Stefan1} and \cite{ChinesischerTyp} well-posedness has been considered for SDEs with a drift consisting of a functional part and a non-functional, locally unbounded part. The strong Feller property has been shown in \cite{Stefan2}.
	
	To prove the strong Feller property for functional, locally unbounded drifts, we use the following convergence concept for random variables which has been invented by the author in \cite{Stefan2}.
	\begin{Theorem}
		\label{goodconv}
		Let $\left(\Omega,\mathcal{F},\Prob\right)$ be some probability space and $\left(E,d\right)$ be a metric space. Furthermore, let $X,X_n:\Omega\to E$, $n\in\N$ be measurable maps. Then the statement
		\begin{enumerate}
			\item
			\begin{enumerate}
				\item$\lim\limits_{n\to\infty}\Prob^*\left(d\left(X,X_n\right)\geq\varepsilon\right)=0 \ \forall\varepsilon>0$,
				\item$\lim\limits_{n\to\infty}\Prob_{X_n}\left(O\right)=\Prob_X\left(O\right)$ for all open $O\subset E$
			\end{enumerate}
		\end{enumerate}
		implies
		\begin{enumerate}
			\item[2.]$\lim\limits_{n\to\infty}\E\abs{f(X)-f(X_n)}=0 \text{ for all bounded, measurable }f:E\to\R$
		\end{enumerate}
		where $\Prob^*$ denotes the outer measure of $\Prob$. Additionally, if there exists some null set $N\subset\Omega$ such that $X(\Omega\setminus N)$ is separable, then the converse implication is also true.
	\end{Theorem}
	\begin{proof}
		See Theorem 1.7 in \cite{Stefan2}.
	\end{proof}
	Although it seems to be a general probabilistic but rather abstract result, it can be applied to the framework of stochastic differential equations. It allows us to show convergence of the Girsanov densities
	$$\lim\limits_{y\to x}\frac{\dd{X^y}_{|[-r,T]}}{\dd{M^y}_{|[-r,T]}}=\frac{\dd{X^x}_{|[-r,T]}}{\dd{M^x}_{|[-r,T]}}\text{ in probability}$$
	for $T>r$ where $M^x$ denotes the solution of equation \eqref{eq} without drift, i.e.
	\begin{align*}
		\dd{M}^x(t)&=\sigma(t,M^x(t))\dd{W}(t)\\
		M_0&=x
	\end{align*}
	Roughly speaking, this is the main step for proving the strong Feller property. For that, Theorem \ref{goodconv} is extremely useful, especially to deal with discontinuous coefficients, which extends the probabilistic approach of Maslowski and Seidler \cite{Maslowski2000}. Remarkably, the coefficients in the drift-free equation do not depend on the past. For proving well-posedness and stability we combine Zvonkin's transformation \cite{Zvonkin} and combine it with the convergence result Theorem \ref{goodconv}. In both cases, we extensively make use of Krylov's estimate for semimartingales \cite{Krylov}. 
	\begin{Notation}
		If not stated otherwise, $W$ will be a $d$-dimensional Brownian motion on some arbitrary but fixed probability space $(\Omega,\mathcal{F},\Prob)$ and every strong solution shall be defined on this space.
		
		However, weak solutions of equation $\eqref{eq}$ might be defined on different filtrated probability spaces. Therefore, we use the short hand notation $(X^x,\tilde{W}^x,\mathds{Q}^x)$ where $X^x$ is an adapted, continuous stochastic process, $\tilde{W}^x$ is an adapted Brownian motion, both with respect to some filtrated probability space $(\tilde{\Omega}^x,\tilde{\mathcal{F}}^x,\mathds{Q}^x,(\tilde{\mathcal{F}}_t)_{t\geq0}^x)$, and $(X^x,\tilde{W}^x)$ solves equation \eqref{eq} with initial value $x$.
	\end{Notation}
	\begin{Condition}
		\label{driftc1}
		For each $T>0$ there exist a measurable $F:[0,T]\times\R^d\to\R^d$ with
		$$\int_{0}^{T}\int_{\R^d}\abs{F(t,x)}^{d+1}\dd{x}\dd{t}<\infty$$
		and $C_1=C_1(T),C_2=C_2(T)\geq0$ with
		$$\int_{0}^{t}\abs{B(s,x)}^2\dd{s}\leq\int_{0}^{t}\abs{F(s,x(s))}\dd{s}+C_1\sup\limits_{s\in[-r,t]}\abs{x(s)}^2+C_2$$
		for all $t\in[0,T]$and $x\in C\left(\R_{\geq-r},\R^d\right)$.
	\end{Condition}
	\begin{Condition}arleta
		\label{sigmac}
		Assume that for all $T>0$ there exists some $C_\sigma=C_\sigma(T)>0$ such that
		\begin{enumerate}
			\item $C_\sigma^{-1}I_{d\times d}\leq\sigma(t,x)\sigma(t,x)^\top\leq C_\sigma I_{d\times d} \ \forall t\in[0,T],x\in\R^d$,
			\item $\norm{\sigma(t,x)-\sigma(t,y)}_{HS}\leq C_\sigma\abs{x-y} \ \forall t\in[0,T],x,y\in\R^d$.
		\end{enumerate}
	\end{Condition}
	\begin{Condition}
		\label{strictpast}
		Assume that there is an $r_{\tilde{B}}\in(0,r)$ such that
		$$B(t,x)=\tilde{B}(t,x)+b(t,x(t))$$
		with $b\in L^{2d+2}\left(\R_{\geq0}\times\R^d;\R^d\right)$ and $\tilde{B}:\R_{\geq0}\times C\left(\R_{\geq-r},\R^d\right)\to\R^d$ measurable where, for fixed $t\geq0$, $\tilde{B}(t,x)$ depends only on $x_{|[-r,t-r_{\tilde{B}}]}$, i.e. 
		$$\tilde{B}(t,x)=\tilde{B}(t,y)\text{ if }x(s)=y(s) \ \forall s\in[-r,t-r_{\tilde{B}}].$$
	\end{Condition}
	\begin{Condition}
		\label{driftc2}
		For $t\in[0,r)$ the function $x\mapsto B(t,x)$ is continuous. Moreover, for each $T>0$ there exist functions $\tilde{F}\in L_{loc}^{d+1}\left([0,T]\times\R^d\right)$ and $G,H:\R_{\geq0}\to\R_{\geq0}$ with $G$ monotone increasing and
		$$\lim\limits_{R\to\infty}\frac{H(R)}{R}=\infty$$
		such that
		$$\int_{0}^{t}H\left(\abs{B(s,x)}^2\right)\dd{s}\leq\int_{0}^{t}\abs{\tilde{F}(s,x(s))}\dd{s}+G\left(\sup\limits_{s\in[-r,t]}\abs{x(s)}\right)$$
		for all $t\in[0,T]$and $x\in C\left(\R_{\geq-r},\R^d\right)$.
	\end{Condition}
	\begin{Notation}
		In the sequel, let $r>0$ be an arbitrary but fixed number and define
		$$\mathcal{C}:=C\left([-r,0],\R^d\right)$$
		equipped with the supremum norm $\norm{\cdot}_\infty$. For a process $X$ defined on $[t-r,t]$ with $t\geq0$, we write
		$$X_t(s):=X(t+s), \ s\in[-r,0].$$
	\end{Notation}
	\begin{Condition}
		\label{boundedmemory}
		The non-anticipating function $B$ has bounded memory, i.e. it holds
		$$B(t,x)=B(t,y)\text{ if }x(s)=y(s) \ \forall s\in [t-r,t].$$
		Then we use the abuse of notation
		$$B(t,x_t)=B(t,x) \ \forall x\in C\left(\R_{\geq-r},\R^d\right)$$
		and similarly for $\tilde{B}$ if \eqref{strictpast} is satisfied.
	\end{Condition}
	The main results read as follows.
	\begin{Theorem}[Existence]
		\label{existence}
		Assume \eqref{driftc1} and \eqref{sigmac}. Then for each initial value $x\in\mathcal{C}$, equation \eqref{eq} has a global weak solution $(X^x,\tilde{W}^x,\mathds{Q}^x)$, which is unique in distribution.
	\end{Theorem}
	\begin{Theorem}[Pathwise Uniqueness]
		\label{pathwiseuniq}
		Assume the localized versions of \eqref{driftc1}, \eqref{sigmac} and \eqref{strictpast}. Then local pathwise uniqueness holds for equation \eqref{eq}, i.e. let $(X^x,W)$ and $(\hat{X}^x,W)$ be two weak solutions of equation \eqref{eq} with initial value $x\in\Co$ on some time interval $[0,\tau]$ for some common Brownian motion $W$ and stopping time $\tau$. Then it follows $X^x=\hat{X}^x$ on $[0,\tau]$ almost surely.
	\end{Theorem}
	\begin{Theorem}[Strong Feller Property]
		\label{StrongFeller}
		Assume \eqref{driftc1}, \eqref{sigmac}, \eqref{driftc2} and \eqref{boundedmemory}. Let $(X^x,\tilde{W}^x,\mathds{Q}^x)$ be weak solutions with initial value $x\in\Co$. Then one has the strong Feller property for all $t>r$, i.e.
		$$\lim\limits_{y\to x}\E _{\mathds{Q}^y}f(X^y_t)=\E_{\mathds{Q}^x}f(X^x_t) \ \forall f\in B_b(\Co).$$
	\end{Theorem}
	\begin{Theorem}[Stability]
		\label{stability}
		Assume \eqref{driftc1}, \eqref{sigmac},\eqref{strictpast}, $\eqref{driftc2}$ and \eqref{boundedmemory}. Let $X^x$ be the strong solutions with initial value $x\in\Co$. Then one has
		$$\lim\limits_{y\to x}\E\norm{X_t^y-X_t^x}_\infty^\gamma=0$$
		for all $0<\gamma<2$ and for $t>r$
		$$\lim\limits_{y\to x}\E\abs{f(X_t^y)-f(X_t^x)}=0 \ \forall f\in B_b(\Co).$$
	\end{Theorem}
	\begin{Remark} \ 
		\begin{enumerate}
			\item Conditions \eqref{driftc1} and \eqref{driftc2} are closed under linear combinations.
			\item Assume, one has
			$$B(t,x_t)=\int_{-r}^{0}k(t,x(t+s))\dd{\mu}(s)$$
			for some Borel measure $\mu$ on $[-r,0]$.
			Then \eqref{driftc1} is fulfilled if $k$ is of at most linear growth in the second variable uniformly on $[0,r)$ and
			$$k\in L^{2d+2}\left([0,T]\times\R^d\right) \ \forall T>0.$$
			If $x\mapsto k(t,x)$ is additionally continuous for $t\in[0,r)$ then condition \eqref{driftc2} will be satisfied. 
			The assumption $\operatorname{supp}\mu\subset[-r,-r_{\tilde{B}}]$ for some $r_{\tilde{B}}\in(0,r)$ implies \eqref{strictpast}.
			\item The continuity assumption in \eqref{driftc2} is not artificial. Consider the following equation
			$$\dd{X^x}(t)=\sgn\left(X^x(t-1)\right)\dd{t}+\dd{W}(t)$$
			with
			$$\sgn(x)=
			\begin{cases}
				1	&	\text{if }x\geq0,\\
				-1	&	\text{if }x<0.
			\end{cases}$$
			Then we have for the strong solutions
			\begin{align*}
				X^0(1)&=W(1)+1,\\
				X^{-1/n}(1)&=W(1)-1-\frac{1}{n}
			\end{align*}
			where we denoted constant paths by real numbers. All conditions are fulfilled but the continuity assumption on the interval $[0,r)$. However, neither the strong Feller property nor convergence in probability hold.
			\item Consider the one-dimensional, deterministic, functional equation
			\begin{align*}
				\dd{x}(t)&=B(t,x(t-1))\dd{t},\\
				x_0&=0
			\end{align*}
			with
			$$
			B(t,z):=
			\begin{cases}
				8t^7				& \text{if \ }t\in[0,1],z\in\R,\\
				\1_{\abs{z}\leq5}\abs{z}^{-1/8}	&	\text{otherwise.}
			\end{cases}
			$$
			The drift $B$ fulfills conditions \eqref{driftc1}, \eqref{strictpast} and \eqref{driftc2}. Computing the solution $x$ yields for $t\in[0,1]$
			$$x(t)=t^8$$
			Consequently, one has to solve
			\begin{align*}
				\dd{x}(t)&=(t-1)^{-1}\dd{t}\\
				x(1)&=1
			\end{align*}
			for $t\in[1,2]$. Integrating both sides yields
			$$x(t)=1+\int_{1}^{t}(s-1)^{-1}\dd{s}=\infty$$
			for each $t\in(1,2)$. It follows that the equation has no global solution in contrast to its regularized version
			$$\dd{X}(t)=B(t,X(t-1))\dd{t}+\dd{W}(t)$$
			although the conditions \eqref{driftc1}, \eqref{strictpast} and \eqref{driftc2} are fulfilled.
		\end{enumerate}
	\end{Remark}
	\section{A-priori Estimates and Existence}
	In the sequel, denote by $M^x$, $x\in\mathcal{C}$ the global, unique strong solution of
	\begin{align*}
		\dd{M^x}(t)&=\sigma\left(t,M^x(t)\right)\dd{W}(t),\\
		M_0^x&=x.
	\end{align*}
	\begin{Notation}
		We denote by $\norm{\cdot}_{op}$ and $\norm{\cdot}_{HS}$ the operator norm and respectively theHilbert-Schmidt norm for matrices $A\in\R^{d\times d}$, i.e.
		$$\norm{A}_{op}=\sup\limits_{v\in\R^d,\abs{v}=1}\abs{Av}, \ \norm{A}_{HS}=\sqrt{\sum_{i,j=1}^{d}\abs{A^{i,j}}^2}.$$
		Additionally, we write for $a,b\in[-\infty,+\infty]$
		$$a\wedge b:=\min\{a,b\}, \ a\vee b:=\max\{a,b\}.$$
	\end{Notation}
	\begin{Remark}
		Condition \eqref{sigmac} implies the following inequalities
		$$\norm{\sigma}_{op}\vee\norm{\sigma^{-1}}_{op}\leq\sqrt{C_\sigma}.$$
	\end{Remark}
	\begin{Lemma}
		\label{Krylov}
		Assume \eqref{sigmac} and let $T>0$, $p>\frac{d+2}{2}$ be given. Then one has for all $0\leq S<T$ and $f\in L_{p}\left([S,T]\times\R^d\right)$ the estimate
		$$\E\left(\int_{S}^{T}f(t,M^x(t))\dd{t}\bigg|\mathcal{F_S}\right)\leq C\norm{f}_{L^p\left([S,T]\times\R^d\right)}$$
		for some constant $C=C(d,p,T,C_\sigma)$. In particular, the constant $C$ is independent of the initial value $x\in\Co$.
	\end{Lemma}
	\begin{proof}
		This follows directly from Theorem 2.1 in \cite{Zhang2011}.
	\end{proof}
	\begin{Lemma}
		\label{expfM}
		Assume \eqref{sigmac}. Then for any $R,T>0$ and $p>\frac{d+2}{2}$ there exists a constant $C_R=C_R(d,p,T,C_\sigma)$ such that
		$$\E\exp\left(\int_{0}^{T}f(t,M^x(t))\dd{t}\right)\leq C_R$$
		for all $f\in L^p\left([0,T]\times\R^d\right)$ with $\norm{f}_{L^p\left([0,T]\times\R^d\right)}\leq R$.
	\end{Lemma}
	\begin{proof}
		See Lemma 2.1 in \cite{Zhang2011}.
	\end{proof}
	\begin{Lemma}
		\label{expsupM}
		Assume $\eqref{sigmac}$. Then for any $T>0$ and $0\leq\alpha<(2dC_\sigma T)^{-1}$, it holds
		$$\E\exp\left(\alpha\sup_{0\leq t\leq T}\abs{M^x(t)}^2\right)\leq\frac{4}{\sqrt{1-2\alpha dC_\sigma T}}\exp\left(\frac{\alpha}{1-2\alpha dC_\sigma T}\abs{x(0)}^2\right).$$
	\end{Lemma}
	\begin{proof}
		See Lemma 2.4 in \cite{Stefan1}.
	\end{proof}
	Let $X$ be a $d$-dimensional It\=o-process of the form
	$$\dd{X}(t)=b(t)\dd{t}+\sigma(t)\dd{W}(t)$$
	where $W$ is a $d$-dimensional Brownian motion. Set
	$$a^{ij}(t):=\frac{1}{2}\sigma(t)\sigma(t)^\top$$
	and let $\tau_R$ be the first exit time of $X(t)$ from the ball $B_R$.
	\begin{Lemma}[Krylov's Estimate]
		For every stopping time $\gamma$ and nonnegative Borel function $f:\R_{\geq0}\times\R^d\to\R_{\geq0}$ one has
		\begin{align*}
			&\E\int_{0}^{\gamma\wedge\tau_R}\left(\det a(t)\right)^{\frac{1}{d+1}}f(t,X(t))\dd{t}\\
			\leq&N(d)\left(\mathds{B}^2+\mathds{A}\right)^{\frac{d}{2(d+1)}}\left(\int_{0}^{\infty}\int_{\abs{x}\leq R}f^{d+1}(t,x)\dd{x}\dd{t}\right)^{\frac{1}{d+1}}
		\end{align*}
		where
		$$\mathds{A}:=\E\int_{0}^{\gamma\wedge\tau_R}\operatorname{tr}a(t)\dd{t}, \ \ \ \mathds{B}:=\E\int_{0}^{\gamma\wedge\tau_R}\abs{b(t)}\dd{t}$$
		and $N(d)$ is a constant depending only on the dimension $d$.
	\end{Lemma}
	\begin{Corollary}
		Assume \eqref{driftc1}, \eqref{sigmac} and let $T>0$. Furthermore, let $(X^x,\tilde{W}^x,\mathds{Q}^x)$ be a solution of equation \eqref{eq} on some time interval $[-r,\tau]$ where $\tau$ is some stopping time with $0\leq\tau\leq T$. Then one has
		$$\E\int_{0}^{\tau}\abs{B(t,X^x)}^2\dd{t}\leq 2C_1\E\left[\sup\limits_{-r\leq t\leq\tau}\abs{X^x(t)}^2\right]+C$$
		where $C=C\left(d,T,C_2,C_\sigma,\norm{F}_{L^{d+1}([0,T]\times\R^d)}\right)$ is some constant.
	\end{Corollary}
	\begin{proof}
		The proof is similar to the one of Corollary 3.2. in \cite{Gyoengy2001}. By Krylov's estimate and Young's inequality, one has
		\begin{align*}
			&\E\int_{0}^{\tau}\abs{B(t,X^x)}^2\dd{t}\\
			\leq&\E\int_{0}^{\tau}\abs{F(t,X^x(t))}\dd{t}+C_1\E\left[\sup\limits_{-r\leq t\leq\tau}\abs{X^x(t)}^2\right]+C_2\\
			\leq& N(d)\left(\mathds{B}^2+\mathds{A}\right)^{\frac{d}{2(d+1)}}\norm{F}_{L^{d+1}([0,T]\times\R^d)}+C_1\E\left[\sup\limits_{-r\leq t\leq\tau}\abs{X^x(t)}^2\right]+C_2\\
			\leq&\frac{1}{2T}\left(\mathds{B}^2+\mathds{A}\right)+C_1\E\left[\sup\limits_{-r\leq t\leq\tau}\abs{X^x(t)}^2\right]+\tilde{C}\\
			\leq&\frac{1}{2}\E\int_{0}^{\tau}\abs{B(t,X^x)}^2\dd{t}+C_1\E\left[\sup\limits_{-r\leq t\leq\tau}\abs{X^x(t)}^2\right]+C
		\end{align*}
		where $C$ and $\tilde{C}$ are constants depending only on $d$, $T$, $C_2$, $C_\sigma$ and $\norm{F}_{L^{d+1}([0,T]\times\R^d)}$ with
		$$\mathds{A}:=\frac{1}{2}\E\int_{0}^{\tau}\operatorname{tr}\left(\sigma(t,X^x(t))\sigma(t,X^x(t))^\top\right)\dd{t}, \ \ \ \mathds{B}:=\E\int_{0}^{\tau}\abs{B(t,X^x)}\dd{t}.$$
	\end{proof}
	\begin{Corollary}
		\label{squareintX}
		Assume \eqref{driftc1}, \eqref{sigmac} and let $T>0$. Furthermore, let $(X^x,\tilde{W}^x,\mathds{Q}^x)$ be a solution of equation \eqref{eq} on some time interval $[-r,\tau]$ where $\tau$ is some stopping time with $0\leq\tau\leq T$. Then one has
		$$\E\left[\sup\limits_{-r\leq t\leq \tau}\abs{X^x(t)}^2\right]\leq C\left(1+\norm{x}_\infty^2\right)$$
		where $C=C\left(d,T,C_1,C_2,C_\sigma,\norm{F}_{L^{d+1}([0,T]\times\R^d)}\right)$ is some constant.
	\end{Corollary}
	\begin{proof}
		Applying Gronwall's lemma and Doob's maximal inequality.
	\end{proof}
	\begin{Corollary}
		\label{KrylovX}
		Assume \eqref{driftc1}, \eqref{sigmac} and let $T>0$. Moreover, let\\$(X^x,\tilde{W}^x,\mathds{Q}^x)$ be a weak solution of equation \eqref{eq} on $[-r,\tau]$ for some stopping time $0\leq\tau\leq T$. Then for any Borel function $f:\R^{d+1}\to\R_{\geq0}$ and $q\geq d+1$, one has
		$$\E\int_{0}^{T\wedge\tau}f(t,X^x(t))\dd{t}\leq N\norm{f}_{L^q(T)}$$
		where $N=N\left(d,T,C_1,C_2,C_\sigma,\norm{F}_{L^{d+1}([0,T]\times\R^d)},\norm{x}_\infty\right)$ is a constant.
	\end{Corollary}
	\begin{proof}
		This follows directly from Krylov's estimate and the Corollaries before.
	\end{proof}
	\begin{Theorem}
		\label{ExUni}
		Assume \eqref{driftc1} and \eqref{sigmac}. Then for every initial values $x\in\Co$, equation \eqref{eq} has a global weak solution. Moreover, for each weak solution $(X^x,\tilde{W}^x,\mathds{Q}^x)$ of equation \eqref{eq} on some time interval $[-r,T]$, $T>0$, one has
		\begin{align*}
		\mathds{Q}_{X^x}^x(A)&=\E_\Prob\bigg[\1_A(M^x)\exp\bigg(\int_{0}^{T}a^x(t)^\top \dd{W}(t)-\frac{1}{2}\int_{0}^{T}\abs{a^x(t)}^2\dd{t}\bigg)\bigg],\\
		a^x(t)&:=\sigma(t,M^x(t))^{-1}B(t,M^x), \ t\in[0,T]
		\end{align*}
		for all measurable $A\subset C([-r,T],\R^d)$.
	\end{Theorem}
	\begin{proof}
		At first, we show the existence of a weak solution. The strong solution $M^x$ is by definition $\left(\mathcal{F}_t\right)_{t\geq0}$-adapted where $\left(\mathcal{F}_t\right)_{t\geq0}$ is the augmented filtration generated by $W$. Next, we construct a probability measure on
		$$\mathcal{F}_\infty:=\sigma\left(\mathcal{F}_t:t\geq0\right)$$
		such that $M^x$ is a global weak solution for equation \eqref{eq}. By Lemma \ref{expfM}, Lemma \ref{expsupM}, conditions \eqref{driftc1} and \eqref{sigmac}, there exist for each $T>0$ a partition $0=T_0\leq T_1,\dots\leq T_{n-1}\leq T_n=T$, $n\in\N$ with	
		$$\E_\Prob\exp\left(\frac{1}{2}\int_{T_{i-1}}^{T_i}\abs{\sigma(t,M^x(t))^{-1}B(t,M^x)}^2\dd{t}\right)<\infty, \ i=1,\dots,n.$$
		Therefore, Novikov's condition is fulfilled for each subinterval, which gives that
		\begin{align*}
		t\mapsto\exp\bigg(&\int_{t\wedge T_{i-1}}^{t\wedge T_i}\left(\sigma(s,M^x(s))^{-1}B(s,M^x)\right)^\top\dd{W}(s)\\
		&-\frac{1}{2}\int_{t\wedge T_{i-1}}^{t\wedge T_i}\abs{\sigma(s,M^x(s))B(s,M^x)}^2\dd{s}\bigg)
		\end{align*}
		is a martingale for $i=1,\dots,n$. Consequently,
		\begin{align*}
		t\mapsto\exp\bigg(&\int_{0}^{t}\left(\sigma(s,M^x(s))^{-1}B(s,M^x)\right)^\top\dd{W}(s)\\
		&-\frac{1}{2}\int_{0}^{t}\abs{\sigma(t,M^x(s))B(t,M^x)}^2\dd{s}\bigg)
		\end{align*}
		is a martingale and by Girsanov's theorem,
		$$\bar{W}(t):=W(t)-\int_{0}^{t}\sigma(s,M^x(s))^{-1}B(s,M^x)\dd{s}, \ t\geq0$$
		is a Brownian motion on $[0,T]$ under the probability measure
		\begin{align*}
		\dd{\bar{\Prob}_T}:=&\exp\bigg(\int_{0}^{T}\left(\sigma(t,M^x(t))^{-1}B(t,M^x)\right)^\top\dd{W}(t)\\
		&-\frac{1}{2}\int_{0}^{T}\abs{\sigma(t,M^x(t))B(t,M^x)}^2\dd{t}\bigg)\dd{\Prob}
		\end{align*}
		and $(M^x,\bar{W},\bar{\Prob}_T)$ is a weak solution of \eqref{eq} on $[-r,T]$ for each $T>0$. Additionally, one has for $0<T_1<T_2$
		$$\bar{\Prob}_{T_1}(A)=\bar{\Prob}_{T_2}(A) \ \forall A\in\mathcal{F}_{T_1},$$
		so the probability measure on $\mathcal{F}_\infty$ uniquely defined by
		$$\bar{\Prob}(A):=\Prob_T(A) \ \forall T>0, A\in\mathcal{F}_T$$
		is indeed well-defined and $(M^x,\bar{W},\bar{\Prob})$ is a global weak solution.
		
		Now, let $(X^x,\tilde{W}^x,\mathds{Q}^x)$ be a weak solution on some time interval $[0,T]$, $T>0$. The following approach is inspired by the techniques used in \cite{liptser2001statistics}. Define
		$$\tau^n(\omega):=\inf\left\{t\geq0:\int_{0}^{t}\abs{B(s,\omega)}^2\dd{s}\geq n\right\}\wedge T, \ \omega\in C([-r,T],\R^d),\ n\in\N.$$
		Then the stopped process $X^{x,n}(t):=X^x(t\wedge\tau^n(X^x))$, $t\in[-r,T]$ fulfills the equation
		\begin{align*}
			\dd{X^{x,n}}(t)=\1_{t\leq\tau^n(X^{x,n})}B(t,X^{x,n})\dd{t}+\1_{t\leq\tau^n(X^{x,n})}\sigma(t,X^{x,n}(t))\dd{\tilde{W}^x}
		\end{align*}
		By construction, Novikov's condition is fulfilled. Consequently, Girsanov's theorem is applicable and
		$$\tilde{W}^{x,n}(t):=\int_{0}^{t\wedge\tau^n(X^{x,n})}\sigma(s,X^{x,n}(s))^{-1}B(s,X^{x,n})\dd{s}+\tilde{W}^x(t), \ t\geq0$$
		is a Brownian motion with respect to the probability measure
		\begin{align*}
			\dd{\mathds{Q}^{x,n}}:=&\exp\bigg(-\int_{0}^{\tau^n(X^{x,n})}\left(\sigma(t,X^{x,n}(t))^{-1}B(t,X^{x,n})\right)^\top\dd{\tilde{W}^x}(t)\\
			&-\frac{1}{2}\int_{0}^{\tau^n(X^{x,n})}\abs{\sigma(t,X^{x,n}(t))B(t,X^{x,n})}^2\dd{t}\bigg)\dd{\mathds{Q}^x}.
		\end{align*}
		The process $X^{x,n}$ solves the equation
		\begin{align*}
			\dd{X^{x,n}}(t)&=\sigma(t,X^{x,n}(t))\dd{\tilde{W}^{x,n}}(t), \ t\in[0,\tau^n(X^{x,n})],\\
			X^{x,n}_0&=x.
		\end{align*}
		Such a solution is (locally) pathwise unique , i.e.
		$$X^{x,n}(t)=M^{x,n}(t), \ t\in[-r,\tau^n(X^{x,n})]$$
		where $M^{x,n}$ is the unique strong solution of
		\begin{align*}
			\dd{M^{x,n}}(t)&=\sigma(t,M^{x,n}(t))\dd{\tilde{W}^{x,n}}(t),\\
			M^{x,n}_0&=x.
		\end{align*}
		and it holds
		$$\tau^n(X^{x,n})=\tau^n(M^{x,n})\text{ a.s.}$$
		Moreover, $\mathds{Q}^x$ and $\mathds{Q}^{x,n}$ are equivalent. Thus,
		\begin{align*}
			&\mathds{Q}^x(X^x\in A)\\
			=&\lim\limits_{n\to\infty}\mathds{Q}^x(\tau^n(X^x)=T,X^x\in A)\\
			=&\lim\limits_{n\to\infty}\mathds{Q}^x(\tau^n(X^{x,n})=T,X^{x,n}\in A)\\
			=&\lim\limits_{n\to\infty}\E_{\mathds{Q}^{x,n}}\bigg[\1_{\tau^n(X^{x,n})=T}\1_A(X^{x,n})\exp\bigg(\int_{0}^{T}\left(\sigma(t,X^{x,n}(t))^{-1}B(t,X^{x,n})\right)^\top\dd{\tilde{W}^{x,n}}(t)\\
			&-\frac{1}{2}\int_{0}^{T}\abs{\sigma(t,X^{x,n}(t))^{-1}B(t,X^{x,n})}^2\dd{t}\bigg)\bigg]\\
			=&\lim\limits_{n\to\infty}\E_{\mathds{Q}^{x,n}}\bigg[\1_{\tau^n(M^{x,n})=T}\1_A(M^{x,n})\exp\bigg(\int_{0}^{T}\left(\sigma(t,M^{x,n}(t))^{-1}B(t,M^{x,n})\right)^\top\dd{\tilde{W}^{x,n}}(t)\\
			&-\frac{1}{2}\int_{0}^{T}\abs{\sigma(t,M^{x,n}(t))^{-1}B(t,M^{x,n})}^2\dd{t}\bigg)\bigg]\\
			=&\lim\limits_{n\to\infty}\E_\Prob\bigg[\1_{\tau^n(M^x)=T}\1_A(M^x)\exp\bigg(\int_{0}^{T}\left(\sigma(t,M^x(t))^{-1}B(t,M^x)\right)^\top\dd{W}^x(t)\\
			&-\frac{1}{2}\int_{0}^{T}\abs{\sigma(t,M^x(t))^{-1}B(t,M^x)}^2\dd{t}\bigg)\bigg]\\
			=&\E_\Prob\bigg[1_A(M^x)\exp\bigg(\int_{0}^{T}\left(\sigma(t,M^x(t))^{-1}B(t,M^x)\right)^\top\dd{W}^x(t)\\
			&-\frac{1}{2}\int_{0}^{T}\abs{\sigma(t,M^x(t))^{-1}B(t,M^x)}^2\dd{t}\bigg)\bigg]
		\end{align*}
		for all measurable $A\subset C([-r,T],\R^d)$.
	\end{proof}
	\begin{Lemma}
		\label{expfX}
		 Assume \eqref{driftc1} with $C_1=0$, \eqref{sigmac} and let $T>0$, $q\geq d+1$ be given. Moreover, let $(X^x,\tilde{W}^x,\mathds{Q}^x)$ be a weak solution of equation \eqref{eq} on $[-r,\tau]$ for some stopping time $0\leq\tau\leq T$. Then one has
		$$\sup\limits_{f\in L^q([0,T]\times\R^d):\norm{f}_{L^q}\leq R}\E_{\mathds{Q}^x}\exp\left(\int_{0}^{\tau}f(t,X^x(t))\dd{t}\right)<\infty$$
		for all $R>0$.
	\end{Lemma}	
	\begin{proof}
		Let
		$$a^x(t):=\sigma(t,M^x(t))^{-1}B(t,M^x), \ t\in[0,T].$$
		Analogous proceeding as in proof of Theorem \ref{ExUni} gives
		\begin{align*}
			&\E_{\mathds{Q}^x}\exp\left(\int_{0}^{\tau}f(t,X^x(t))\dd{t}\right)\\
			\leq&\E_{\Prob}\exp\bigg(\int_{0}^{T}f(t,M^x(t))\dd{t}+\int_{0}^{T}a^x(t)^\top\dd{W}(t)-\frac{1}{2}\int_{0}^{T}\abs{a^x(t)}^2\dd{t}\bigg)\\
			\leq&\left[\E_{\Prob}\exp\left(\int_{0}^{T}2f(t,M^x(t))\dd{t}\right)\right]^\frac{1}{2}\bigg[\E_\Prob\exp\bigg(2\int_{0}^{T}a^x(t)^\top\dd{W}(t)-\int_{0}^{T}\abs{a^x(t)}^2\dd{t}\bigg)\bigg]^\frac{1}{2}\\
			\leq&\left[\E_{\Prob}\exp\left(\int_{0}^{T}2f(t,M^x(t))\dd{t}\right)\right]^\frac{1}{2}
			\cdot\left[\E_\Prob\exp\left(6\int_{0}^{T}\abs{a^x(t)}^2\dd{t}\right)\right]^\frac{1}{4}.
		\end{align*}
		The uniform bound follows from condition \eqref{driftc1} and Lemma \ref{expfM}.
	\end{proof}
	\begin{Lemma}
		\label{expsupX}
		 Assume \eqref{driftc1} with $C_1=0$, \eqref{sigmac} and and let $T>0$ be given. Let $(X^x,\tilde{W}^x,\mathds{Q}^x)$ be a weak solution of equation \eqref{eq} on $[-r,\tau]$ for some stopping time $0\leq\tau\leq T$. Then the following inequality holds.
		\begin{align*}
			&\E_{\mathds{Q}^x}\exp\left(\alpha\sup\limits_{-r\leq t\leq\tau}\abs{X^x(t)}^2\right)\\
			\leq&\frac{C}{\sqrt[4]{1-4\alpha dC_\sigma T}}\exp\left(\frac{\alpha}{1-4\alpha dC_\sigma T}\norm{x}_\infty^2\right)
		\end{align*}
		for all $0\leq\alpha<(4dC_\sigma T)^{-1}$ and a constant $C=C\left(d,T,C_2,C_\sigma,\norm{F}_{L^{d+1}([0,T]\times\R^d)}\right)$.
	\end{Lemma}
	\begin{proof}
		As before, let
		$$a^x(t):=\sigma(t,M^x(t))^{-1}B(t,M^x), \ t\in[0,T].$$
		By the assumed conditions and Lemma \ref{expfM}, one has
		\begin{align*}
			&\E_\Prob\exp\left(6\int_{0}^{T}\abs{a^x(t)}^2\dd{t}\right)\\
			\leq&K_1\E_\Prob\exp\left(6C_\sigma\int_{0}^{T}\abs{B(t,M^x)}^2\dd{t}\right)\\
			\leq&K_2\E_\Prob\exp\left(6C_\sigma\int_{0}^{T}\abs{F(t,M^x(t))}\dd{t}\right)\\
			\leq&K_3
		\end{align*}
		for constants $K_1$, $K_2$ and $K_3$ that only depend on $d$, $T$, $C_2$, $C_\sigma$, $\norm{F}_{L^{d+1}([0,T]\times\R^d)}$. By Theorem \ref{ExUni}, one obtains
		\begin{align*}
			&\E_{\mathds{Q}^x}\left(\alpha\sup\limits_{-r\leq t\leq\tau}\abs{X^x(t)}^2\right)\\
			\leq&\E_\Prob\exp\left(\alpha\sup\limits_{-r\leq t\leq T}\abs{M^x(t)}^2+\int_{0}^{T}a^x(t)^\top\dd{W}(t)-\frac{1}{2}\int_{0}^{T}\abs{a^x(t)}^2\dd{t}\right)\\
			\leq&\left[\E_\Prob\exp\left(2\alpha\sup\limits_{-r\leq t\leq T}\abs{M^x(t)}^2\right)\right]^{\frac{1}{2}}\cdot\left[\E_\Prob\exp\left(2\int_{0}^{T}a^x(t)^\top\dd{W}(t)-\int_{0}^{T}\abs{a^x(t)}^2\dd{t}\right)\right]^{\frac{1}{2}}\\
			\leq&\left[\E_\Prob\exp\left(2\alpha\sup\limits_{-r\leq t\leq T}\abs{M^x(t)}^2\right)\right]^{\frac{1}{2}}\cdot\left[\E_\Prob\exp\left(6\int_{0}^{T}\abs{a^x(t)}^2\dd{t}\right)\right]^{\frac{1}{4}}\\
			\leq&\frac{C}{\sqrt[4]{1-4\alpha dC_\sigma T}}\exp\left(\frac{\alpha}{1-4\alpha dC_\sigma T}\norm{x}_\infty^2\right)
		\end{align*}
		for a constant $C=C\left(d,T,C_2,C_\sigma,\norm{F}_{L^{d+1}([0,T]\times\R^d)}\right)$.
	\end{proof}
	\section{Strong Feller Property}
	The following theorem is a consequence of a log-Harnack inequality that has been shown in \cite{WangYuan} and requires the Lipschitz-continuity of $\sigma$ in space.
	\begin{Theorem}
		\label{martfeller}
		Assume \eqref{sigmac}. Then one has for all $t>r$
		$$\lim\limits_{y\to x}\E f(M_t^y)=\E f(M^x_t) \ \forall f\in B_b(\Co).$$
	\end{Theorem}
	\begin{Lemma}
		\label{bconv}
		Assume \eqref{driftc1}, \eqref{sigmac}, \eqref{driftc2} and \eqref{boundedmemory}. Then one has
		$$\lim\limits_{y\to x}\Prob\left(\int_{0}^{T}\abs{B(t,M^x_t)-B(t,M^y_t)}^2\dd{t}>\varepsilon\right)=0 \ \forall\varepsilon>0.$$
	\end{Lemma}
	\begin{proof}
		By Theorems \ref{martfeller} and \ref{goodconv}, one has for all $t>r$
		$$\lim\limits_{y\to x}\E\abs{f(M^x_t)-f(M^y_t)}=0 \ \forall f\in B_b(\Co).$$
		Therefore, one has for all $f\in B_b\left([0,T]\times\Co\right)$
		$$\lim\limits_{y\to x}\E\int_{r}^{T}\abs{f(t,M^x_t)-f(t,M^y_t)}\dd{t}=0.$$
		Consequently,
		$$\lim\limits_{y\to x}\Prob\otimes\lambda_{|[r,T]}\left(\abs{B(\cdot,M^y_{\cdot})-B(\cdot,M^x_{\cdot})}>\varepsilon\right)=0 \ \forall\varepsilon>0.$$
		By condition \eqref{driftc2}, one has also
		$$\lim\limits_{y\to x}\Prob\otimes\lambda_{|[0,r]}\left(\abs{B(\cdot,M^y_{\cdot})-B(\cdot,M^x_{\cdot})}>\varepsilon\right)=0 \ \forall\varepsilon>0.$$
		Therefore, it holds
		$$\lim\limits_{y\to x}\Prob\otimes\lambda_{|[0,T]}\left(\abs{B(\cdot,M^y_{\cdot})-B(\cdot,M^x_{\cdot})}>\varepsilon\right)=0 \ \forall\varepsilon>0.$$
		Now, define for $R>0$
		$$B^R\left(t,x\right):=\1_{\sup\limits_{-r\leq s\leq t}\abs{x(s)}^2<R}B\left(t,x_t\right), \ x\in C\left(\R_{\geq-r},\R^d\right)$$
		Each $B^R$, $R>0$ fulfills condition \eqref{driftc1} and the suitably modified version of \eqref{driftc2} with bounded $G$. It follows
		$$\sup\limits_{y\in\Co}\E\int_{0}^{T}H\left(\abs{B^R(t,M^y)}^2\right)\dd{t}<\infty, \ R>0$$
		by Lemma \ref{Krylov}, \ref{expsupM} and condition \eqref{driftc2}. Hence, $\left\{\abs{B^R(\cdot,M^y_\cdot)}^2:y\in\Co\right\}$ is uniformly integrable for each $R>0$. Since 
		$$\lim\limits_{y\to x}\E\norm{M^y_t-M^x_t}_\infty \ \forall t\geq0,$$
		it holds
		$$\lim\limits_{y\to x}\Prob\otimes\lambda_{|[0,T]}\left(\abs{B^R(\cdot,M^y)-B^R(\cdot,M^x)}>\varepsilon\right)=0 \ \forall\varepsilon>0.$$
		Furthermore,
		$$\lim\limits_{R\to\infty}\sup\limits_{x\in\Co,\norm{x}\leq\tilde{R}}\Prob\left(\sup\limits_{t\in[-r,T]}\abs{M^x(t)}\geq R\right)=0 \ \forall \tilde{R}>0$$
		holds by condition \eqref{sigmac}. Thus,
		\begin{align*}
			&\limsup\limits_{y\to x}\Prob\left(\int_{0}^{T}\abs{B(t,M^x(t))-B(t,M^y(t))}^2\dd{t}>\varepsilon\right)\\
			\leq&\limsup\limits_{y\to x}\Prob\left(\int_{0}^{T}\abs{B^R(t,M^x(t))-B^R(t,M^y(t))}^2\dd{t}>\varepsilon\right)\\
			&+2\sup\limits_{z\in\Co,\norm{z}\leq2\norm{z}_\infty}\Prob\left(\sup\limits_{t\in[-r,T]}\abs{M^z(t)}\geq R\right)\\
			=&2\sup\limits_{z\in\Co,\norm{z}\leq2\norm{z}_\infty}\Prob\left(\sup\limits_{t\in[-r,T]}\abs{M^z(t)}\geq R\right)
		\end{align*}
		Now, one can let $R\to\infty$, which proofs the claim.
	\end{proof}
	\begin{proof}[Proof of Theorem \ref{StrongFeller}]
		Let $t>r$ and $f\in B_b(\Co)$, then one has by Theorem \ref{ExUni}
		\begin{align*}
			&\E_{\mathds{Q}^x}f(X_t^x)-\E_{\mathds{Q}^y}f(X_t^y)\\
			=&\E_\Prob(D^x(t)f(M_t^x))-\E_\Prob(D^y(t)f(M_t^y))\\
			=&\E_\Prob[D^x(t)(f(M_t^x)-f(M_t^y))]+\E_\Prob[(D^x(t)-D^y(t))f(M_t^y)]\\
			\leq&\E_\Prob[D^x(t)(f(M_t^x)-f(M_t^y))]+\norm{f}_\infty\E_\Prob\abs{D^x(t)-D^y(t)}
		\end{align*}
		where we define for every $z\in\Co$
		\begin{align*}
			a^z(t)&:=\sigma(t,M^z(t))^{-1}B(t,M_t^z),\\
			D^z(t)&:=\exp\left(\int_{0}^{t}a^z(s)^\top\dd{W}(s)-\frac{1}{2}\int_{0}^{t}\abs{a^z(s)}^2\dd{s}\right).
		\end{align*}
		By condition \eqref{sigmac}, It\=o's formula and the stochastic Gronwall Lemma \ref{Gronwall}, it holds
		\begin{align}
			\label{convprobM}
			\lim\limits_{y\to x}\Prob\left(\abs{M^y_t-M^x_t}>\varepsilon\right)=0 \ \forall\varepsilon>0.
		\end{align}
		Applying Theorems \ref{goodconv} and \ref{martfeller}, gives
		$$\lim\limits_{y\to x}\E_\Prob\abs{f(M_t^y)-f(M_t^x)}=0$$
		and in particular,
		$$\lim\limits_{y\to x}\Prob\left(\abs{D^x(t)f(M_t^{x_n})-D^x(t)f(M_t^x)}>\varepsilon\right)=0 \ \forall\varepsilon>0.$$
		By the dominated convergence theorem, it follows
		$$\lim\limits_{y\to x}\E_\Prob[D^x(t)(f(M_t^y)-f(M_t^x))]=0.$$
		Consequently, it remains to show that
		$$\lim\limits_{y\to x}\E_\Prob\abs{D^y(t)-D^x(t)}=0.$$
		Since one has $\E_\Prob D^z(t)=1$ for all $z\in\Co$, it suffices to show
		$$\lim\limits_{y\to x}\Prob\left(\abs{D^y(t)-D^x(t)}>\varepsilon\right)=0 \ \forall\varepsilon>0$$
		by standard measure theoretic arguments. Therefore, it is sufficient to show
		$$\lim\limits_{y\to x}\Prob\left(\int_{0}^{t}\abs{a^y(s)-a^x(s)}^2\dd{s}>\varepsilon\right)=0 \ \forall \varepsilon>0$$
		by the martingale isometry. One has
		\begin{align*}
			&\int_{0}^{t}\abs{a^y(s)-a^x(s)}^2\dd{s}\\
			\leq&2\int_{0}^{t}\norm{\sigma\left(s,M^y(s)\right)^{-1}-\sigma\left(s,M^x(s)\right)^{-1}}^2_{op}\abs{B(s,M^x_s)}^2\dd{s}\\
			&+2C_\sigma\int_{0}^{t}\abs{B(s,M^y_s)-B(s,M^x_s)}^2\dd{t}.
		\end{align*}
		The second term converges to zero by the assumed conditions and Lemma \ref{bconv}. Moreover,
		$$\lim\limits_{y\to x}\Prob\otimes\lambda_{|[0,t]}\left(\norm{\sigma\left(\cdot,M^y(\cdot)\right)^{-1}-\sigma\left(\cdot,M^x(\cdot)\right)^{-1}}_{op}>\varepsilon\right)=0 \ \forall\varepsilon>0$$
		holds by \eqref{convprobM}, the continuity of $\sigma$ in space and the continuity of the inverting map $A\mapsto A^{-1}$ on the space of invertible matrices. Additionally, one can bound the first integrand by
		$$2C_\sigma\abs{B(\cdot,M^x_\cdot)}^2,$$
		which is $\Prob\otimes\lambda_{|[0,t]}$-integrable by Lemma \ref{Krylov}. Consequently, one can apply the dominated convergence theorem and the proof is complete.
	\end{proof}
	\section{Pathwise Uniqueness and Stability}
	\begin{Notation}
		We introduce - as in \cite{Zhang2011}- the following function space. For $p,\in(1,\infty)$ and $0\leq S<T$, denote by $W_p^{1,2}\left([S,T]\times\R^d\right)$ the closure of compactly supported, smooth functions on $[S,T]\times\R^d$ with respect to the norm
		$$\norm{u}_{W_p^{1,2}\left([S,T]\times\R^d\right)}:=\norm{\partial_tu}_{L^p[S,T]}+\norm{u}_{L^p([S,T];W^{2,p})}, \ u\in C_c^\infty([S,T]\times\R^d).$$
	\end{Notation}
	Let $p:=2d+2$. By Theorem \ref{PDE}, for every $0<T\leq T_0$, there exists a solution
	$$\tilde{u}(\cdot;T)\in\left(W_{p}^{1,2}\left([0,T_0]\times\R^d\right)\right)^d$$
	of the coordinatewise PDE system
	\begin{align*}
		\partial_t\tilde{u}(t,x;T)+L_t\tilde{u}(t,x;T)+b(t,x)&=0,\\
		\tilde{u}(T,x;T)&=0
	\end{align*}
	for all $t\in[0,T]$ and $x\in\R^d$ where
	$$L_tv(t,x):=\frac{1}{2}\sum_{i,j,k=1}^{d}\sigma^{i,k}(t,x)\sigma^{j,k}(t,x)\partial_i\partial_jv(t,x)+b(t,x)\cdot\nabla v(t,x), \ v\in W_p^{1,2}\left([0,T_0]\times\R^d\right).$$
	Additionally, it holds
	$$\sup\limits_{T\in[0,T_0]}\norm{\tilde{u}^i(\cdot;T)}_{W_p^{1,2}\left([0,T]\times\R^d\right)}<\infty, \ i=1,\dots,d$$
	and by the embedding Theorem \ref{embedding}, there exists a uniform $\delta$ such that for all $0\leq S\leq T$ with $T-S\leq\delta$
	$$\abs{\tilde{u}(t,x;T)-\tilde{u}(t,y;T)}\leq\frac{1}{2}\abs{x-y}$$
	for all $t\in[S,T]$ and $x,y\in\R^d$. Furthermore, the function
	$$u(t,x;T):=\tilde{u}(t,x;T)+x$$
	satisfies coordinatewise the equation
	\begin{align*}
		\partial_tu(t,x;T)+L_tu(t,x;T)&=0,\\
		u(T,x;T)&=x.
	\end{align*}
	\begin{proof}[Proof of Theorem \ref{pathwiseuniq}]
		Let $(X^x,W)$ and $(\hat{X}^x,W)$ be two weak solutions of equation \eqref{eq} with initial value $x\in\Co$ for some common Brownian motion $W$ on the time interval $[0,\tau]$ for some stopping time $\tau$. By localization, we can assume that condition \eqref{driftc1} is fulfilled with $C_1=0$ and that $\tau$ is bounded by some $T_0>0$. Choose $\delta>0$ like above with the additional restraint $\delta<r_{\tilde{B}}$. By induction, it suffices to prove for every $0\leq S\leq T\leq T_0$ with $T-S\leq\delta$
		\begin{align*}
			&X_{|[-r,S\wedge\tau]}^x=\hat{X}^x_{|[-r,S\wedge\tau]}\\
			\implies&X_{T\wedge\tau}^x=\hat{X}_{T\wedge\tau}^x.
		\end{align*}
		For the sake of simplicity, we write $u(\cdot):=u(\cdot;T)$. Furthermore, define
		\begin{align*}
			Y(t)&:=u(t,X(t)), \ S\wedge\tau\leq t\leq T\wedge\tau,\\
			\hat{Y}(t)&:=u(t,\hat{X}^x(t)), \ S\wedge\tau\leq t\leq T\wedge\tau.
		\end{align*}
		By the choice of $\delta$, one has for the difference processes $Z(t):=X^x(t)-\hat{X}^x(t)$ and $\tilde{Z}(t):=Y(t)-\hat{Y}(t)$
		$$\frac{1}{2}\abs{\tilde{Z}(t)}\leq\abs{Z(t)}\leq\frac{3}{2}\abs{\tilde{Z}(t)}, \ S\wedge\tau\leq t\leq T\wedge\tau.$$
		Due to Lemma \ref{expfX}, Lemma \ref{Ito} is applicable, which gives for $S\wedge\tau\leq t\leq T\wedge\tau$
		\begin{align*}
			\tilde{Z}(t)=&\int_{S}^{t}\left(Du(s,X^x(s))\tilde{B}(s,X^x)-D u(s,\hat{X}^x(s))\tilde{B}(s,\hat{X}^x)\right)\dd{s}\\
			&+\int_{S}^{t}\left(Du(s,X^x(s))\sigma(s,X^x(s))-Du(s,\hat{X}^x(s))\sigma(s,\hat{X}^x(s))\right)\dd{W}(s)\\
			=&\int_{S}^{t}\left(Du(s,X^x(s))\tilde{B}(s,X^x)-D u(s,\hat{X}^x(s))\tilde{B}(s,X^x)\right)\dd{s}\\
			&+\int_{S}^{t}\left(Du(s,X^x(s))\sigma(s,X^x(s))-Du(s,\hat{X}^x(s))\sigma(s,\hat{X}^x(s))\right)\dd{W}(s)
		\end{align*}
		and consequently
		\begin{align*}
			&\dd{\abs{\tilde{Z}}^2}(t)\\
			=&2\tilde{Z}(t)^\top\left(Du(t,X^x(t))\tilde{B}(t,X^x)-Du(t,\hat{X}^x(t))\tilde{B}(t,X^x)\right)\dd{t}\\
			&+2\tilde{Z}(t)^\top\left(Du(t,X^x(t))\sigma(t,X^x(t))-Du(t,X^{x_n}(t))\sigma(t,X^{x_n}(t))\right)\dd{W}(t)\\
			&+\norm{Du(t,X^x(t))\sigma(t,X^x(t))-Du(t,\hat{X}^x(t))\sigma(t,\hat{X}^x(t))}_{HS}^2\dd{t}
		\end{align*}
		Using the boundedness of $Du$, condition \eqref{driftc1} and Young's inequality gives for $S\leq t_1\leq t_2\leq T$
		\begin{align*}
			&\abs{\tilde{Z}(t_2)}^2-\abs{\tilde{Z}(t_1)}^2\\
			\leq&\int_{t_1}^{t_2}\abs{\tilde{Z}(s)}\norm{Du(s,X^x(s))-Du(s,X^{x_n}(s))}_{op}\abs{\tilde{B}(s,X^x)}\dd{s}\\
			+&c\int_{t_1}^{t_2}\tilde{Z}(s)^\top\left(Du(s,X^x(s))\sigma(s,X^x(s))-Du(s,X^{x_n}(s))\sigma(s,X^{x_n}(s))\right)\dd{W}(s)\\
			+&c\int_{t_1}^{t_2}\norm{Du(s,X^x(s))\sigma(s,X^x(s))-Du(s,X^{x_n}(s))\sigma(s,X^{x_n}(s))}_{HS}^2\dd{s}
		\end{align*}
		where $c>0$ is a constant. As in \cite{Fedrizzi2}, one can use a suitable multiplier of the form $\exp(-A(t))$ where $A$ is an adapted, continuous process. Here, we choose
		\begin{align*}
			A(t)&:=c\int_{S}^{t}\abs{\tilde{B}(s,X^x)}\frac{\norm{Du(s,X^x(s))-Du(s,\hat{X}^x(s))}_{op}}{\abs{\tilde{Z}(s)}}\1_{\tilde{Z}(s)\neq0}\dd{s}\\
			&+c\int_{S}^{t}\frac{\norm{Du(s,X^x(s))\sigma(s,X^x(s))-Du(s,\hat{X}^x(s))\sigma(s,\hat{X}^x(s))}_{HS}^2}{\abs{\tilde{Z}(s)}^2}\1_{\tilde{Z}(s)\neq0}\dd{s}
		\end{align*}
		for $S\leq t\leq T$.
		To show that $A$ is indeed well defined - namely finite - it suffices to show
		$$\E \exp\left(\frac{1}{2}A(T)\right)<\infty.$$
		Since $u$ belongs coordinatewise to $ W_p^{1,2}\left([0,T_0]\times\R^d\right)$ and by conditions \eqref{sigmac}, it holds
		$$(Du\cdot\sigma)^{i,j}\in L^p\left(T_0;W^{1,p}\left(\R^d\right)\right), \ i,j=1,\dots,d.$$
		Additionally, $C^\infty_c\left(\R^{d+1}\right)$ is dense in $L^p\left(T_0;W^{1,p}\left(\R^d\right)\right)$. Hence, by Young's inequality, Lemma \ref{expfX} and Lemma \ref{expsupX}, it suffices to show for all $\tilde{R}>0$ the existence of a constant $C_{R}>0$ such that
		$$\E\exp\left(\int_{S}^{T}\frac{\abs{f(s,X^x(s))- f(s,\hat{X}^x(s))}^2}{\abs{\tilde{Z}(s)}^2}\1_{\tilde{Z}(s)\neq0}\dd{s}\right)\leq C_{\tilde{R}}$$ 
		for all $f\in C^\infty\left(\R^{d+1}\right)$ with $\norm{f}_{L^p\left(T_0;W^{1,p}\left(\R^d\right)\right)}\leq R$. By Lemmas \ref{expfX} and \ref{Hardy-Littlewood}, one obtains
		\begin{align*}
			&\E\exp\left(\int_{S}^{T}\frac{\abs{f(s,X^x(s))- f(s,\hat{X}^x(s))}^2}{\abs{\tilde{Z}(s)}^2}\1_{\tilde{Z}(s)\neq0}\dd{s}\right)\\
			\leq&\E\exp\left(C_d^2\int_{S}^{T}\left(\mathcal{M}\abs{\nabla f}(X^x(s))+\mathcal{M}\abs{\nabla f}(\hat{X}^x(s))\right)^2\dd{s}\right)\\
			<\infty.
		\end{align*}
		Now, it holds for $S\leq t\leq T$
		\begin{align*}
			e^{-A(t)}\abs{\tilde{Z}(t)}^2\leq&c\int_{S}^{t}e^{-A(s)}\abs{\tilde{Z}(s)}^2\dd{s}+\text{local martingale}
		\end{align*}
		by the It\=o formula. Applying the stochastic Gronwall Lemma \ref{Gronwall} gives
		$$\E\left[\sup_{t\in[S,T]}e^{-\frac{1}{2}A(t)}\abs{\tilde{Z}(t)}\right]=0,$$
		which finishes the proof.
	\end{proof}
	The following result is a rather technical one, which will be used to proof Theorem \ref{stability}.
	\begin{Proposition}
		\label{BXconv}
		Assume \eqref{driftc1}, \eqref{sigmac}, \eqref{strictpast}, \eqref{driftc2} and \eqref{boundedmemory}. Furthermore, let $X^x$, $x\in\Co$ be the strong solutions to equation \eqref{eq} with initial value $x$ and assume that
		\begin{align}
			\label{XSconv}
			\lim\limits_{y\to x}\Prob\left(\norm{X_S^y-X^x_S}_\infty>\varepsilon\right)=0 \ \forall\varepsilon>0,\forall x\in\Co
		\end{align}
		for some $S\geq0$. Then one has for each $R>0$
		$$\lim\limits_{n\to\infty}\E\int_{S\wedge\tau_{R}^{x,y}}^{(S+r_{\tilde{B}})\wedge\tau_{R}^{x,y}}\abs{\tilde{B}(s,X_s^y)-\tilde{B}(s,X_s^x)}^2\dd{s}=0$$
		where
		$$\tau_{R}^{x,y}:=\sup\left\{t\geq0:\sup\limits_{-r\leq s\leq t}\abs{X^x(s)}^2<R,\sup\limits_{-r\leq s\leq t}\abs{X^y(s)}^2<R\right\}.$$
	\end{Proposition}
	\begin{proof}
		By condition \eqref{strictpast}, one can write
		\begin{align}
			\label{rewriting}
			\tilde{B}(t,X_t^x)=g(t,X_S^x) \ t\in[S,S+r_{\tilde{B}}],x\in\Co.
		\end{align}
		If $S>r$, Theorem \ref{StrongFeller} gives
		$$\lim\limits_{y\to x}\E f(X_S^y)=\E f(X_S^x) \ \forall f\in B_b(\Co).$$
		Consequently, combining it with \eqref{XSconv}, \eqref{rewriting} and Theorem \ref{goodconv} gives
		\begin{align}
			\label{dummyconv}
			\lim\limits_{y\to x}\Prob\left(\abs{g(t,X_S^y)-g(t,X_S^x)}>\varepsilon\right)=0 \ \forall\varepsilon>0,t\in[S,S+r_{\tilde{B}}].
		\end{align}
		If $S\leq r$, one can use the continuity assumption \eqref{driftc2} and \eqref{XSconv} to deduce \eqref{dummyconv}, too. Therefore
		$$\lim\limits_{y\to x}\Prob\otimes\lambda_{|[S,S+r_{\tilde{B}}]}\left(\abs{B(\cdot,X_\cdot^y)-B(\cdot,X_\cdot^x)}^2>\varepsilon\right)=0 \ \forall\varepsilon>0.$$
		It follows
		$$\lim\limits_{y\to x}\Prob\otimes\lambda_{|[S,S+r_{\tilde{B}}]}\left(\1_{\tau_{R}^{x,y}\leq \cdot}\abs{B(\cdot,X_\cdot^y)-B(\cdot,X_\cdot^x)}^2>\varepsilon\right)=0 \ \forall\varepsilon>0.$$
		Now, one can use Lemma \ref{Krylov} and condition \eqref{driftc2} to obtain
		$$\sup\limits_{y\in\Co}\E\int_{S\wedge\tau_{R}^{x,y}}^{(S+r_{\tilde{B}})\wedge\tau_{R}^{x,y}}H\left(\abs{B(t,X_t^y)}^2\right)\dd{t}<\infty,$$
		which guarantees the uniform integrability of $\left\{\1_{\tau_{R}^{y,y}<\cdot}\abs{B(\cdot,M^y_\cdot)}^2:y\in\Co\right\}$ with respect to the measure $\Prob\otimes\lambda_{|[S,S+r_{\tilde{B}}]}$.
	\end{proof}
	\begin{proof}[Proof of Theorem \ref{stability}]
		Choose $\delta>0$ like before with the additional restraint $\delta<r_{\tilde{B}}$. By induction and Lemma \eqref{expsupX}, it suffices to prove for every $0\leq S\leq T\leq T_0$ with $T-S\leq\delta$ the implication
		\begin{align*}
			&\lim\limits_{y\to x}\E\norm{X_S^y-X_S^x}_\infty^\gamma=0 \ \forall x\in\Co,0<\gamma<2\\
			\implies&\lim\limits_{y\to x}\Prob\left(\sup\limits_{s\in[S,T]}\abs{X^y(s)-X^x(s)}>\varepsilon\right)=0 \ \forall \varepsilon>0,x\in\Co.
		\end{align*}
		For the sake of simplicity, we write $u(\cdot):=u(\cdot;T)$. Furthermore, define
		\begin{align*}
			Y^x(t)&:=u(t,X(t)), \ S\leq t\leq T,\\
			Y^y(t)&:=u(t,X^y(t)), \ S\leq t\leq T.
		\end{align*}
		By the choice of $\delta$, one has for the difference processes $Z(t):=X^x(t)-X^y(t)$ and $\tilde{Z}(t):=Y^x(t)-Y^y(t)$
		$$\frac{1}{2}\abs{\tilde{Z}(t)}\leq\abs{Z(t)}\leq\frac{3}{2}\abs{\tilde{Z}(t)}, \ S\leq t\leq T.$$
		Due to Lemma \ref{expfX}, Lemma \ref{Ito} is applicable, which gives
		\begin{align*}
			\tilde{Z}(t)=&\int_{S}^{t}\left(Du(s,X^x(s))\tilde{B}(s,X^x_s)-D u(s,X^y(s))\tilde{B}(s,X^y_s)\right)\dd{s}\\
			&+\int_{S}^{t}\left(Du(s,X^x(s))\sigma(s,X^x(s))-D u(s,X^y(s))\sigma(s,X^y(s))\right)\dd{W}(s)
		\end{align*}
		and consequently
		\begin{align*}
			&\dd{\abs{\tilde{Z}}^2}(t)\\
			=&2\tilde{Z}(t)^\top\left(Du(t,X^x(t))\tilde{B}(t,X^x_t)-Du(t,X^y(t))\tilde{B}(t,X^y_t)\right)\dd{t}\\
			&+2\tilde{Z}(t)^\top\left(Du(t,X^x(t))\sigma(t,X^x(t))-Du(t,X^y(t))\sigma(t,X^y(t))\right)\dd{W}(t)\\
			&+\norm{Du(t,X^x(t))\sigma(t,X^x(t))-Du(t,X^y(t))\sigma(t,X^y(t))}_{HS}^2\dd{t}
		\end{align*}
		Using the boundedness of $Du$ and Young's inequality gives for $S\leq t_1\leq t_2\leq T$
		\begin{align*}
			&\abs{\tilde{Z}(t_2)}^2-\abs{\tilde{Z}(t_1)}^2\\
			\leq&c\int_{t_1}^{t_2}\abs{\tilde{Z}(s)}^2\dd{s}\\
			+&c\int_{t_1}^{t_2}\abs{\tilde{B}(s,X_s^x)-\tilde{B}(s,X_s^y)}^2\dd{s}\\
			+&c\int_{t_1}^{t_2}\abs{\tilde{Z}(s)}\norm{Du(s,X^x(s))-Du(s,X^y(s))}_{op}\abs{\tilde{B}(s,X^x_s)}\dd{s}\\
			+&c\int_{t_1}^{t_2}\tilde{Z}(s)^\top\left(Du(s,X^x(s))\sigma(s,X^x(s))-Du(s,X^y(s))\sigma(s,X^y(s))\right)\dd{W}(s)\\
			+&c\int_{t_1}^{t_2}\norm{Du(s,X^x(s))\sigma(s,X^x(s))-Du(s,X^y(s))\sigma(s,X^y(s))}_{HS}^2\dd{s}
		\end{align*}
		where $c>0$ is a constant. Like before, we one can use the multiplier $\exp\left(-A(t)\right)$ where
		\begin{align*}
			A(t)&:=c\int_{S}^{t}\abs{B(s,X^x_s)}\frac{\norm{Du(s,X^x(s))-Du(s,X^y(s))}_{op}}{\abs{\tilde{Z}(s)}}\1_{\tilde{Z}(s)\neq0}\dd{s}\\
			&+c\int_{S}^{t}\frac{\norm{Du(s,X^x(s))\sigma(s,X^x(s))-Du(s,X^y(s))\sigma(s,X^y(s))}_{HS}^2}{\abs{\tilde{Z}(s)}^2}\1_{\tilde{Z}(s)\neq0}\dd{s}
		\end{align*}
		for $S\leq t\leq T$.
		Again, one has
		$$\E \exp\left(\frac{1}{2}A(T)\right)\leq \hat{C}$$
		where $\hat{C}$ is some constant not depending on $x,y\in\Co$. By the It\=o formula, it holds for $S\leq t\leq T$
		\begin{align*}
			e^{-A(t)}\abs{\tilde{Z}(t)}^2\leq&\abs{\tilde{Z}(S)}^2+c\int_{S}^{t}\abs{\tilde{B}(s,X_s^x)-\tilde{B}(s,X_s^y)}^2\dd{s}\\
			&+c\int_{S}^{t}e^{-A(s)}\abs{\tilde{Z}(s)}^2\dd{s}+\text{local martingale}.
		\end{align*}
		Applying the stochastic Gronwall Lemma \ref{Gronwall} gives
		$$\E\left[\sup_{t\in[S,T]}e^{-\frac{1}{2}A(t)}\abs{\tilde{Z}(t)}\right]\leq \tilde{C}\E\abs{\tilde{Z}(S)}+\tilde{C}\E\left(\int_{S}^{T}\abs{\tilde{B}(s,X_s^x)-\tilde{B}(s,X_s^y)}^2\dd{s}\right)^{\frac{1}{2}}$$
		for a constant $\tilde{C}$ which does not depend on $x,y\in\Co$. By Lemma \ref{squareintX},
		$$\lim\limits_{R\to\infty}\sup\limits_{z\in\Co,\norm{z}_\infty\leq2\norm{x}_\infty}\Prob\left(\sup\limits_{-r\leq t\leq T}\abs{X^x(t)}^2>R\right)=0.$$
		holds. Thus, applying Lemma \ref{BXconv} and the induction hypothesis gives
		$$\lim\limits_{y\to x}\Prob\left(\int_{S}^{T}\abs{\tilde{B}(s,X_s^x)-\tilde{B}(s,X_s^y)}^2\dd{s}>\varepsilon\right)=0 \ \forall\varepsilon>0.$$
		By Corollary \ref{squareintX}, one has
		$$\sup\limits_{z\in\Co,\norm{z}_\infty\leq2\norm{x}_\infty}\E\int_{S}^{T}\abs{\tilde{B}(s,X_s^x)-\tilde{B}(s,X_s^y)}^2\dd{s}<\infty$$
		and consequently,
		$$\lim\limits_{y\to x}\E\left(\int_{S}^{T}\abs{\tilde{B}(s,X_s^x)-\tilde{B}(s,X_s^y)}^2\dd{s}\right)^{\frac{1}{2}}=0.$$
	\end{proof}
	\appendix
	\section{Appendix}
	\begin{Theorem}
		\label{PDE}
		Assume \eqref{sigmac} and $b\in L^p\left([0,T]\times\R^d\right)$ with $p>d+2$. Then for any $T>0$ and $f\in L^p\left([0,T]\times\R^d\right)$, there exists a unique solution $u\in W_p^{1,2}\left([0,T]\times\R^d\right)$ of the following PDE
		\begin{align*}
		\partial_tu(t,x)+L_tu(t,x)+f(t,x)&=0,\\
		u(T,x)&=0
		\end{align*}
		with the bound
		$$\norm{u}_{W_p^{1,2}\left([S,T]\times\R^d\right)}\leq C\norm{f}_{L^p\left([S,T]\times\R^d\right)}$$
		for any $S\in[0,T]$ and some constant $C=C(T,C_\sigma,p,\norm{b}_{L^p\left([0,T]\times\R^d\right)})>0$.
	\end{Theorem}
	\begin{proof}
		See Theorem 10.3 in \cite{KrylovRoeckner}.
	\end{proof}
	\begin{Theorem}
		\label{embedding}
		Let $p\in(1,\infty)$, $T>0$ and $u\in W_p^{1,2}\left([0,T]\times\R^d\right)$.
		\begin{enumerate}
			\item If $p>\frac{d+2}{2}$, then $u$ is a bounded H\"older continuous function on $[0,T]\times\R^d$ and for any $0<\varepsilon$, $\delta\leq1$ satisfying
			$$\varepsilon+\frac{d+2}{p}<2, \ \ 2\delta+\frac{d+2}{p}<2,$$
			there exists a constant $N=N(p,\varepsilon,\delta)$ such that
			\begin{align*}
			\abs{u(t,x)-u(s,x)}&\leq N\abs{t-s}^\delta\norm{u}^{1-\frac{1}{q}-\delta}_{L^p\left(T;\mathds{W}^{2,p}\left(\R^d\right)\right)}\norm{\partial_tu}^{\frac{1}{p}+\delta}_{L^p\left([0,T]\times\R^d\right)},\\
			\abs{u(t,x)}+\frac{\abs{u(t,x)-u(t,y)}}{\abs{x-y}^\varepsilon}&\leq NT^{-\frac{1}{p}}\left(\norm{u}_{L^p\left(T;\mathds{W}^{2,p}\left(\R^d\right)\right)}+T\norm{\partial_tu}_{L^p\left([0,T]\times\R^d\right)}\right)
			\end{align*}
			for all $s,t\in[0,T]$ and $x,y\in\R^d,x\neq y$.
			\item If $p>d+2$, then $\nabla u$ is a bounded H\"older continuous function on $[0,T]\times\R^d$ and for any $\varepsilon\in(0,1)$ satisfying
			$$\varepsilon+\frac{d+2}{p}<1,$$
			there exists a constant $N=N(p,\varepsilon)$ such that
			\begin{align*}
			\abs{\nabla u(t,x)-\nabla u(s,x)}&\leq N\abs{t-s}^\delta\norm{ u}^{1-\frac{1}{p}-\frac{\varepsilon}{2}}_{L^p\left(T;\mathds{W}^{2,p}\left(\R^d\right)\right)}\norm{\partial_tu}^{\frac{1}{p}+\frac{\varepsilon}{2}}_{L^p\left([0,T]\times\R^d\right)},\\
			\abs{\nabla u(t,x)}+\frac{\abs{\nabla u(t,x)-\nabla u(t,y)}}{\abs{x-y}^\varepsilon}&\leq NT^{-\frac{1}{p}}\left(\norm{u}_{L^p\left(T;\mathds{W}^{2,p}\left(\R^d\right)\right)}+T\norm{\partial_tu}_{L^p\left([0,T]\times\R^d\right)}\right)
			\end{align*}
			for all $s,t\in[0,T]$ and $x,y\in\R^d,x\neq y$.
		\end{enumerate}
	\end{Theorem}
	\begin{proof}
		See \cite[p. 22, 23, 36]{Fedrizzi}.
	\end{proof}
	In the next lemma we identify every $u\in W_p^{1,2}$ with its regular version.
	\begin{Lemma}[It\=o formula for $W_p^{1,2}$-functions]
		\label{Ito}
		Let $T>0$, $p>d+2$. Let $X:\Omega\times[0,T]\to\R^d$ be a semimartingale on some filtrated probability space $\left(\Omega,\mathcal{F},\Prob,\left(\mathcal{F}_t\right)_{t\geq0}\right)$ of the form
		$$\dd{X}(t)=b(t)\dd{t}+\sigma(t)\dd{W}(t)$$
		where $W$ is a $d$-dimensional Brownian motion, $b:\Omega\times[0,T]\to\R^d$ and $\sigma:\Omega\times[0,T]\to\R^{d\times d}$ are progressively measurable with
		$$\Prob\left(\norm{b}_{L^1[0,T]}+\norm{a^{i,j}}_{L^\delta[0,T]}<\infty\right)=1, \ i,j=1,\dots,d$$
		for some $1<\delta\leq\infty$ where $a:=\sigma\sigma^\top$. Furthermore, assume that there exists a constant $C>0$ with
		$$\E\int_{0}^{T}f(t,X(t))\dd{t}\leq C\norm{f}_{L^{p/\delta*}\left([0,T]\times\R^d\right)}$$
		for all $f\in L^{p/\delta^*}\left([0,T]\times\R^d\right)$ where $\delta^*$ denotes the conjugate exponent of $\delta$.
		Then for any $u\in W_p^{1,2}\left([0,T]\times\R^d\right)$, the It\=o formula holds, i.e.
		\begin{align*}
			u(t,X(t))-u(0,X(0))=&\int_{0}^{t}\partial_tu(s,X(s))\dd{s}+\int_{0}^{t}\nabla u(s,X(s))^\top b(s)\dd{s}\\
			&+\int_{0}^{t}\nabla u(s,X(s))^\top\sigma(s)\dd{W}(s)\\
			&+\frac{1}{2}\sum_{i,j=1}^{d}\int_{0}^{t}\partial_i\partial_j u(s,X(s))a^{i,j}(s)\dd{s}.
		\end{align*}
	\end{Lemma}
	\begin{proof}
		See \cite{Stefan1}.
	\end{proof}
	Let $\phi$ be a locally integrable function on $\R^d$. The Hardy-Littlewood maximal function is defined by
	$$\mathcal{M}\phi(x):=\sup_{0<r<\infty}\frac{1}{\abs{B_r}}\int_{B_r}\phi(x+y)\dd{y}$$
	where $B_r$ is the Euclidean ball of radius $r$. The following result is cited from Appendix A in \cite{Crippa}.
	\begin{Lemma} \ 
		\label{Hardy-Littlewood}
		\begin{enumerate}
			\item	There exists a constant $C_d>0$ such that for all $\phi\in C^{\infty}\left(\R^d\right)$ and $x,y\in\R^d$,
			$$\abs{\phi(x)-\phi(y)}\leq C_d\abs{x-y}\left(\mathcal{M}\abs{\nabla\phi}(x)+\mathcal{M}\abs{\nabla\phi}(y)\right).$$
			\item For any $p>1$, there exists a constant $C_{d,p}$ such that for all $\phi\in L^p\left(\R^d\right)$,
			$$\norm{\mathcal{M}\phi}_{L^p}\leq C_{d,p}\norm{\phi}_{L^p.}$$
		\end{enumerate}
	\end{Lemma}
	For a real-valued process denote $Y^*(t):=\sup\limits_{0\leq s\leq t}Y(s)$.
	\begin{Lemma}
		\label{Gronwall}
		Let $Z$ and $H$ be nonnegative, adapted processes with continuous paths and assume that $\psi$ is nonnegative and progressively measurable. Let $M$ be a continuous local martingale starting at 0. If
		$$Z(t)\leq\int_{0}^{t}\psi(s)Z(s)\dd{s}+M(t)+H(t)$$
		holds for all $t\geq0$, then for $p\in(0,1)$ and $\mu,\nu>1$ such that $\frac{1}{\mu}+\frac{1}{\nu}=1$ and $p\nu<1$, we have
		$$\E\sup\limits_{0\leq s\leq t}Z(s)^p\leq(c_{p\nu}+1)^{1/\nu}\left(\E\exp\left\{p\mu\int_{0}^{t}\psi(s)\dd{s}\right\}\right)^{1/\mu}\left(\E(H^*(t))^{p\nu}\right)^{1/\nu}$$
		where
		$$c_p:=\left(4\wedge\frac{1}{p}\right)\frac{\pi p}{\sin(\pi p)}.$$
		If $\psi$ is deterministic, then
		$$\E\sup\limits_{0\leq s\leq t}Z(s)^p\leq(1+c_p)\exp\left\{p\int_{0}^{t}\psi(s)\dd{s}\right\}\left(\E(H^*(t))^p\right)$$
		and
		$$\E Z(t)\leq\exp\left\{\int_{0}^{t}\psi(s)\dd{s}\right\}\E H^*(t).$$
	\end{Lemma}
	\begin{proof}
		See \cite{Scheutzow}.
	\end{proof}
	\bibliographystyle{plain}
	\bibliography{Bibliography}

\begin{thebibliography}{10}

\bibitem{Stefan1}
S.~Bachmann.
\newblock Well-posedness and stability for a class of stochastic delay
  differential equations with singular drift.
\newblock {\em Stochastics and Dynamics}, 0(0):1850019, 0.

\bibitem{Stefan2}
S.~{Bachmann}.
\newblock {On the Strong Feller Property of Stochastic Delay Differential
  Equations with Singular Drift}.
\newblock {\em ArXiv e-prints}, September 2017.

\bibitem{Crippa}
Gianluca Crippa and Camillo De~Lellis.
\newblock Estimates and regularity results for the {D}i{P}erna-{L}ions flow.
\newblock {\em J. Reine Angew. Math.}, 616:15--46, 2008.

\bibitem{es-sarhir2009}
A.~Es-Sarhir, M.-K. von Renesse, and M.~Scheutzow.
\newblock Harnack inequality for functional sdes with bounded memory.
\newblock {\em Electron. Commun. Probab.}, 14:560--565, 2009.

\bibitem{Fedrizzi}
E.~Fedrizzi.
\newblock Uniqueness and flow theorems for solutions of {SDE}s with low
  regularity of the drift.
\newblock {\em tesi di Laurea in Matematica, Università di Pisa}, 2009.

\bibitem{Fedrizzi2}
E.~Fedrizzi and F.~Flandoli.
\newblock Pathwise uniqueness and continuous dependence of {SDE}s with
  non-regular drift.
\newblock {\em Stochastics}, 83(3):241--257, 2011.

\bibitem{Gyoengy2001}
I.~Gy{\"o}ngy and T.~Mart{\'{\i}}nez.
\newblock On stochastic differential equations with locally unbounded drift.
\newblock {\em Czechoslovak Math. J.}, 51(126)(4):763--783, 2001.

\bibitem{ChinesischerTyp}
X.~Huang.
\newblock Strong solutions for functional sdes with singular drift.
\newblock {\em Stochastics and Dynamics}, 0(0):1850015, 0.

\bibitem{Krylov}
N.~V. Krylov.
\newblock Estimates of the maximum of the solution of a parabolic equation and
  estimates of the distribution of a semimartingale.
\newblock {\em Mat. Sb. (N.S.)}, 130(172)(2):207--221, 284, 1986.

\bibitem{KrylovRoeckner}
N.~V. Krylov and M.~R{\"o}ckner.
\newblock Strong solutions of stochastic equations with singular time dependent
  drift.
\newblock {\em Probab. Theory Related Fields}, 131(2):154--196, 2005.

\bibitem{liptser2001statistics}
R.~S. Liptser and A.~N. Shiryaev.
\newblock {\em Statistics of Random Processes: I. General Theory}, pages
  286--297.
\newblock Springer, 2001.

\bibitem{Maslowski2000}
B.~Maslowski and J.~Seidler.
\newblock Probabilistic approach to the strong feller property.
\newblock {\em Probability Theory and Related Fields}, 118(2):187--210, Oct
  2000.

\bibitem{Portenko}
N.~I. Portenko.
\newblock {\em Generalized diffusion processes}, volume~83 of {\em Translations
  of Mathematical Monographs}.
\newblock American Mathematical Society, Providence, RI, 1990.
\newblock Translated from the Russian by H. H. McFaden.

\bibitem{Scheutzow}
M.~Scheutzow.
\newblock A stochastic gronwall lemma.
\newblock {\em Infinite Dimensional Analysis, Quantum Probability and Related
  Topics}, 16(02):1350019, 2013.

\bibitem{Ver}
A.~Ju. Veretennikov.
\newblock Strong solutions of stochastic differential equations.
\newblock {\em Teor. Veroyatnost. i Primenen.}, 24(2):348--360, 1979.

\bibitem{WangYuan}
F.-Y. Wang and C.~Yuan.
\newblock Harnack inequalities for functional {SDE}s with multiplicative noise
  and applications.
\newblock {\em Stochastic Process. Appl.}, 121(11):2692--2710, 2011.

\bibitem{Zhang2011}
X.~Zhang.
\newblock Stochastic homeomorphism flows of {SDE}s with singular drifts and
  {S}obolev diffusion coefficients.
\newblock {\em Electron. J. Probab.}, 16:no. 38, 1096--1116, 2011.

\bibitem{Zvonkin}
A.~K. Zvonkin.
\newblock A transformation of the phase space of a diffusion process that will
  remove the drift.
\newblock {\em Mat. Sb. (N.S.)}, 93(135):129--149, 152, 1974.

\end{thebibliography}
\end{document}